[1994/12/01]
\documentclass{ijmart-mod}
\chardef\bslash=`\\ 

\usepackage{graphicx}
\usepackage[breaklinks=true]{hyperref}
\usepackage{mathtools}
\usepackage{complexity}

\newtheorem{thm}{Theorem}[section]

\newtheorem{lem}[thm]{Lemma}
\newtheorem{prop}[thm]{Proposition}

\theoremstyle{definition}
\newtheorem{defn}[thm]{Definition}
\newtheorem{rem}[thm]{Remark}

\theoremstyle{remark}

\newcommand{\eval}[2][\right]{\relax
  \ifx#1\right\relax \left.\fi#2#1\rvert}

\begin{document}

\title{The rotation distance of brooms}

\author[J. Cardinal]{Jean Cardinal}
\address{Universit{\'e} libre de Bruxelles (ULB), Brussels, Belgium}
\email{jean.cardinal@ulb.be}

\author[L. Pournin]{Lionel Pournin}
\address{Universit{\'e} Paris 13, Villetaneuse, France}
\email{lionel.pournin@univ-paris13.fr}

\author[M. Valencia]{Mario Valencia-Pabon}
\address{Universit{\'e} de Lorraine, Nancy, France}
\email{mario.valencia@loria.fr}

\begin{abstract}
  The associahedron $\mathcal{A}(G)$ of a graph $G$ has the property that its vertices can be thought of as the search trees on $G$ and its edges as the rotations between two search trees. If $G$ is a simple path, then $\mathcal{A}(G)$ is the usual associahedron and the search trees on $G$ are binary search trees. Computing distances in the graph of $\mathcal{A}(G)$, or equivalently, the rotation distance between two binary search trees, is a major open problem. Here, we consider the different case when $G$ is a complete split graph. In that case, $\mathcal{A}(G)$ interpolates between the stellohedron and the permutohedron, and all the search trees on $G$ are brooms. We show that the rotation distance between any two such brooms and therefore the distance between any two vertices in the graph of the associahedron of $G$ can be computed in quasi-quadratic time in the number of vertices of $G$.
\end{abstract}
\maketitle

\section{Introduction}\label{CPV2.sec.0}

Given any graph $G$, one can build a polytope $\mathcal{A}(G)$, the \emph{graph associahedron of $G$} \cite{CarrDevadoss2006,Devadoss2009,Postnikov2009} that encodes the combinatorics of certain objects related to $G$. 
When the graph $G$ is a path on $n$ vertices, this polytope is the ubiquitous associahedron~\cite{CeballosSantosZiegler2015,Lee1989,Stasheff1963a,Stasheff1963b,Tamari1951}, whose vertices can be represented indifferently by the triangulations of a convex polygon with $n+2$ vertices or the binary trees with $n$ internal nodes. With different families of graphs, we retrieve other well-known polytopes: the associahedron of a complete graph is the permutohedron~\cite{Bowman1972,GuilbaudRosenstiehl1963}, the associahedron of a cycle is the cyclohedron~\cite{BottTaubes1994,Simion2003}, and the associahedron of a star is the stellohedron~\cite{PostnikovReinerWilliams2008}. In general, when $G$ is an arbitrary graph, the vertices of $\mathcal{A}(G)$ can be shown to correspond to the \emph{search trees} on $G$ (that are also sometimes called \emph{elimination trees}), and the edges of $\mathcal{A}(G)$ to the \emph{rotations} between search trees \cite{CardinalMerinoMutze2021,MannevillePilaud2015}.

The graphs consisting of the vertices and edges of a graph associahedron have been attracting a lot of attention due to their applications to data structures~\cite{BerendsohnKozma2022,BoseCardinalIaconoKoumoutsosLangerman2020,SleatorTarjanThurston1988}, sampling algorithms~\cite{EppsteinFrishberg2022}, and phylogenetics~\cite{CulikWood1982,SempleSteel2003}, among other subjects. The diameter and hamiltonicity properties of associahedra~\cite{LucasRoelantsvanBaronaigienRuskey1993,Pournin2014,SleatorTarjanThurston1988}, cyclohedra \cite{Pournin2017}, tree associahedra \cite{BoseCardinalIaconoKoumoutsosLangerman2020,CardinalLangermanPerez-Lantero2018}, chordal graph associahedra \cite{CardinalMerinoMutze2021,CardinalMerinoMutze2022}, caterpillar associahedra~\cite{Berendsohn2022}, and complete multipartite graph associahedra~\cite{CardinalPourninValencia2022} have been considered. General results on the diameter and hamiltonicity of all graph associahedra have also been obtained in \cite{CardinalPourninValencia2022,MannevillePilaud2015}.

Many of the above mentioned results consist in bounding the diameter of graph associahedra, or equivalently the largest possible rotation distance between any two search trees. 
In this paper, we are interested in the computational complexity of the following problem: Given a graph $G$ and two search trees on $G$, what is the rotation distance between them?

In a recent breakthrough paper~\cite{IKKKMNO23}, it is shown that the corresponding decision problem is \NP-complete, by reduction from the balanced minimum $(s,t)$\nobreakdash-cut in a graph. The question remains open in the case of the usual associahedron, or equivalently, when $G$ is restricted to be the $n$-vertex path: It is not known whether the rotation distance between two binary search trees on $n$ nodes can be computed in time polynomial in $n$, or whether the problem is \NP\nobreakdash-hard~\cite{BarilPallo2006,ClearyStJohn2010,ClearyMaio2018}. In that case, the problem can also be cast as that of computing the flip distance between two triangulations of a convex polygon. It is known to be fixed-parameter tractable: For any constant $k$, we can decide in linear time if the flip distance is at most $k$~\cite{ClearyStJohn2009,Lucas2010,KanjSedgwickXia2017,LiXia2022}. Some geometric generalizations, involving triangulations of simple polygons or point sets in the plane, are also known to be \NP-hard~\cite{AichholzerMulzerPilz2015,LubiwPathak2015,Pilz2014}, while the flip distance can be computed in polynomial time for certain families of point sets~\cite{Eppstein2010}.

When $G$ is a complete graph, hence when $\mathcal{A}(G)$ is a permutohedron, the rotation distance between two search trees is simply the number of inversions between the two corresponding permutations, a quantity that can be computed using classical sorting algorithms.

Our main result is that for a family of graph associahedra that interpolates between the permutohedron (when $G$ is a complete graph) and the stellohedron (when $G$ is a star), the problem of determining the distance between two vertices, hence the rotation distance between two search trees on $G$, is solvable in polynomial time. To our knowledge, this is the first non-trivial example of a family of generalized permutohedra~\cite{Postnikov2009}, besides standard permutohedra, for which geodesics (shortest paths in their $1$-skeletons) can be computed in polynomial time in the dimension of the polytope. In particular, our algorithm solves the problem for stellohedra, a question which was also open\footnote{The question was discussed for instance at the Dagstuhl Seminar 22062 in February 2022 (\url{https://www.dagstuhl.de/22062}).}. The family of graph associahedra we consider are the ones whose underlying graph is $G$ is a complete split graph: $G$ is obtained from the complete graph by selecting a non-empty proper subset $Q$ of its vertices and by removing all the edges between two vertices from $Q$. If $Q$ is a singleton, then $G$ is complete and $\mathcal{A}(G)$ is the permutohedron. On the other hand, if $Q$ contains all the vertices of $G$ but one, then $G$ is a star and $\mathcal{A}(G)$ is the stellohedron. It turns out that when $G$ is a complete split graph, the search trees on $G$ are \emph{brooms}: they are formed from a simple path, one end of which is the root while the other end is attached to the leaves \cite{CardinalPourninValencia2022}. Therefore, our main result can be rephrased as follows.

\begin{thm}\label{CPV2.sec.0.thm.1}
The rotation distance between two brooms on a complete split graph with $n$ vertices can be computed in time $O(n^{2+o(1)})$.
\end{thm}

When $G$ is a star and $\mathcal{A}(G)$ is the stellohedron, the problem has an interesting reformulation in terms of \emph{partial permutations}, defined as ordered subsets of $\{1,\ldots,n\}$. One can equip partial permutations with three types of elementary operations. The first operation consists in exchanging two consecutive elements of the partial permutation (a transposition of adjacent elements in the terminology of permutations). The second operation consists in adding at the end of the partial permutation an element that it does not contain yet. The third operation is the inverse of the second one and amounts to removing the last element of the partial permutation. It follows from Theorem~\ref{CPV2.sec.0.thm.1} that the number of these operations required to transform a partial permutation of $\{1,\ldots,n\}$ into another can be determined in polynomial time in $n$. 

As observed by Santos (see for instance Section 5.4 in~\cite{CP16}), the associahedron of the star with $n$ leaves is also known to be combinatorially equivalent to the secondary polytope of a point set in dimension $n-1$ formed by two dilated copies of the standard $(n-1)$-dimensional simplex. This point set has been referred to as the ''mother of all examples'' for its role in the theory of regular triangulations~\cite{LRS10}. Therefore, another consequence of Theorem~\ref{CPV2.sec.0.thm.1} is that the flip distance between regular triangulations of this family of point sets can be computed in polynomial time.  

Our strategy for proving Theorem \ref{CPV2.sec.0.thm.1} consists of modeling the rotation distance between two brooms on a complete split graph $G$ as a quadratic function of $0/1$ variables, and then minimizing this function. While minimizing a quadratic function of $0/1$ variables is \NP-hard in general, we shall see that our model is among the ones for which the problem is known to be polynomial. 

The remainder of the paper is organized as follows. In Section \ref{CPV2.sec.0.5}, we recall how the boundary complex of the graph associahedron $\mathcal{A}(G)$ is described is terms of the tubings on the underlying graph $G$ as well as the precise correspondence between the vertices and edges of $\mathcal{A}(G)$ and the search trees on $G$ and their rotations. In Section~\ref{CPV2.sec.1}, we describe our quadratic model for the rotation distance between brooms. In Section~\ref{CPV2.sec.2}, we analyze this model and its solutions. We also show in that section that the rotation distance of brooms on a complete split graph can be computed in polynomial time. Finally, we show that this complexity is in fact at worst quasi-quadratic in Section~\ref{CPV2.sec.3}. Interestingly, this is proved by reduction to a minimum $(s,t)$-cut problem.

\section{Graph associahedra, search trees, and brooms}\label{CPV2.sec.0.5}

In this section we consider a connected simple graph $G$ and we describe the face complex of $\mathcal{A}(G)$, the graph associahedron of $G$. We first give a description of that face complex in terms of tubings of $G$ \cite{AguiarArdila2017,CarrDevadoss2006,MannevillePilaud2015}, and then proceed with a description of the $1$-skeleton of $G$ in terms of search trees \cite{CardinalMerinoMutze2021}. We recall that a subset of a set $S$ is called \emph{proper} when it is distinct from $S$.

\begin{defn}
A \emph{tubing} on $G$ is a collection $\mathcal{T}$ of non-empty proper subsets of vertices of $G$, each of whose is referred to as a \emph{tube}, such that
\begin{enumerate}
\item[(i)] for every $S$ in $\mathcal{T}$, the subgraph induced by $S$ in $G$ is connected and
\item[(ii)] for every two distinct tubes $S_1$ and $S_2$ in $\mathcal{T}$, either
\begin{enumerate}
\item[(ii.a)] $S_1$ is a subset of $S_2$ or $S_2$ a subset of $S_1$, in which case we say that $S_1$ and $S_2$ are \emph{nested}, or
\item[(ii.b)] the subgraph of $G$ induced by $S_1\cup S_2$ is not connected, in which case we say that $S_1$ and $S_2$ are \emph{non-adjacent}.
\end{enumerate}
\end{enumerate}
\end{defn}

Four tubings $\mathcal{T}_1$ to $\mathcal{T}_4$ on a claw $G$ (a star with three leaves) with leaves labeled by $1$, $2$, and $3$ are shown in Figure \ref{CPV2.sec.0.5.fig.0}. The tubes are depicted in this figure as curves around the vertices they contain (whose labels, that are omitted, can be recovered from the sketch of $G$ on the left of the figure). For instance,
$$
\mathcal{T}_1=\Bigl\{\{z\},\{3,z\},\{2, 3,z\}\Bigr\}\mbox{,}
$$
where $z$ denotes the center of the star and
$$
\mathcal{T}_2=\Bigl\{\{3\},\{3,z\},\{2, 3,z\}\Bigr\}\mbox{.}
$$

The graph associahedron $\mathcal{A}(G)$ is then defined from the tubings of $G$: $\mathcal{A}(G)$ is any polytope whose face lattice is isomorphic to the reverse inclusion order of the tubings of $G$~\cite{AguiarArdila2017,CarrDevadoss2006,MannevillePilaud2015}. It can be realized geometrically in different ways, for instance as a Minkowski sum of standard simplices~\cite{Devadoss2009,Postnikov2009}.

Now denote by $n$ the number of vertices of $G$. By definition, the vertices of $\mathcal{A}(G)$ are in one-to-one correspondence with inclusion-wise maximal tubings of $G$ or, equivalently, the tubings of $G$ of size exactly $n-1$. For instance, the four tubings on the claw shown in Figure \ref{CPV2.sec.0.5.fig.0} correspond to vertices as they each contain three tubes.
In fact, more generally, the $k$-dimensional faces of $\mathcal{A}(G)$ are in one-to-one correspondence with the tubings of $G$ of size $n-k-1$ and $\mathcal{A}(G)$ itself has dimension $n-1$. Thus, the facets of $\mathcal{A}(G)$ are in one-to-one correspondance with the connected induced subgraphs of $G$ other that $G$ itself. It can be shown that any tubing of size $n-2$ can be completed into a maximal tubing in exactly two distinct ways~\cite{MannevillePilaud2015}. We can therefore interpret the 1\nobreakdash-skeleton of the graph associahedron as a flip graph on maximal tubings, where a flip consists of replacing one tube within a tubing of size $n-1$ by the unique other tube such that the resulting set is still a tubing of $G$. In Figure~\ref{CPV2.sec.0.5.fig.0}, each of the represented tubings is related to the next by a flip. For example, $\mathcal{T}_1$ can be transformed into $\mathcal{T}_2$ by replacing $\{z\}$ with $\{3\}$. Likewise, $\mathcal{T}_3$ can be transformed into $\mathcal{T}_4$ by replacing $\{z,2,3\}$ with $\{1\}$.
\begin{figure}
\begin{centering}
\includegraphics[scale=1]{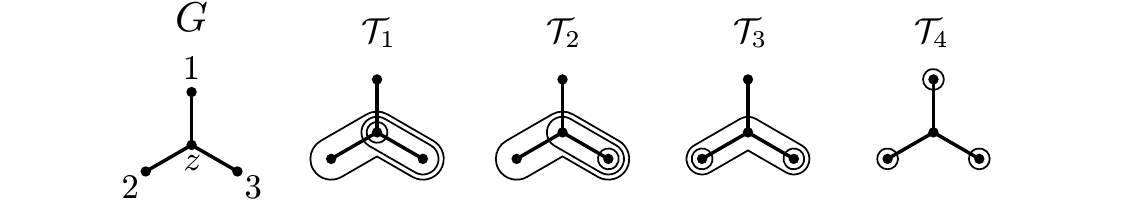}
\caption{A claw $G$ (left) and four tubings of $G$.}\label{CPV2.sec.0.5.fig.0}
\end{centering}
\end{figure}

A simpler interpretation of maximal tubings, and of the vertices of the graph associahedron, can be given in terms of search trees.

\begin{defn}
A \emph{search tree} $T$ on $G$ is any rooted tree that shares its vertices with $G$ and can be obtained from the following recursive procedure.
\begin{enumerate}
\item[(i)] If $G$ has only one vertex, then $T$ consists of that single vertex.
\item[(ii)] Otherwise, pick any vertex $r$ of $G$ as the root of the search tree. The search tree $T$ consists of that root attached to an (unordered) collection of subtrees defined recursively as search trees on each of the connected components of the graph $G-r$.
\end{enumerate}
\end{defn}

We can associate a maximal tubing $\mathcal{T}$ of $G$ to any search tree $T$ on $G$ by taking, for the tubes in $\mathcal{T}$, the vertex sets of all the subtree of $T$ (except for $T$ itself). It is an easy exercise to check that the inclusion-wise maximal tubings of $G$ are in one-to-one correspondence with search trees on $G$. For example, the search trees on the claw that correspond to the four tubings from Figure \ref{CPV2.sec.0.5.fig.0} are shown on the right of Figure \ref{CPV2.sec.0.5.fig.1}.

Furthermore, the flip operations on maximal tubings can be interpreted as \emph{rotations} on search trees. A rotation is the local change observed on a search tree by reversing the removal order between a vertex and one of its children in the tree~\cite{CardinalMerinoMutze2022}.

\begin{defn}
  Given a search tree $T$ on $G$ and two vertices $u,v$ such that $u$ is the parent of $v$ in $T$, the \emph{rotation} around the edge $u,v$ of $T$ is the operation that transforms $T$ into another search tree in which:
  \begin{itemize}
  \item $v$ becomes the parent of $u$, and the parent of $u$ in $T$, if any, becomes the parent of $v$,
  \item a subtree $S$ of $v$ in $T$ is reattached to $u$ if its vertices belong to the same connected component of $G_u-v$ as $u$, where $G_u$ is the subgraph induced by the subtree of $T$ rooted at $u$,
  \end{itemize}
  and all other edges of $T$ are preserved.
\end{defn}

Two search trees on $G$ are adjacent in the $1$\nobreakdash-skeleton of $\mathcal{A}(G)$ if and only if they differ by a rotation. Search trees are also sometimes referred to as \emph{elimination trees}~\cite{CardinalMerinoMutze2021}, and more generally as $\mathcal{B}$-trees, where $\mathcal{B}$ is the graphical building set of $G$~\cite{PostnikovReinerWilliams2008}.

We can define a rotation distance problem between two search trees $T_1$ and $T_2$ on $G$ as the minimum number of rotations needed to transform $T_1$ into $T_2$.
\begin{figure}
\begin{centering}
\includegraphics[scale=1]{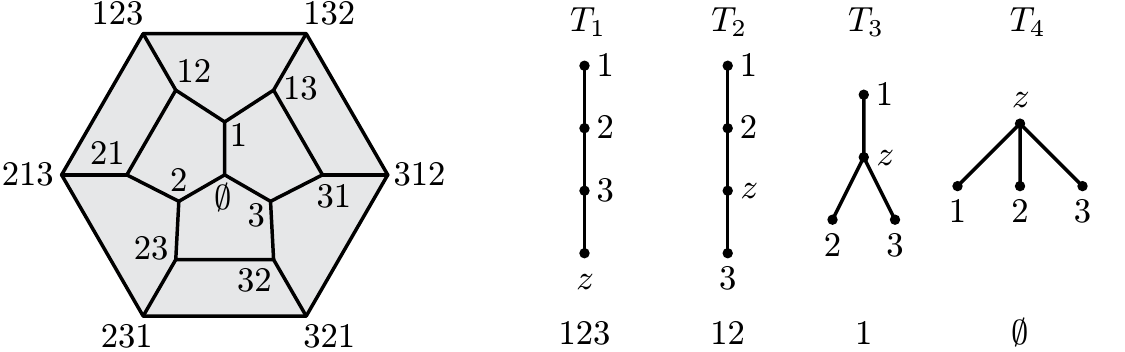}
\caption{The stellohedron $\mathcal{A}(G)$ when $G$ is the claw (left), whose vertices are labeled using the partial permutations of $\{1,2,3\}$, and four brooms on $G$ (right), each one above the partial permutation it corresponds to.}\label{CPV2.sec.0.5.fig.1}
\end{centering}
\end{figure}
This generalizes the usual rotation distance on binary search trees, which are precisely the search trees in the case when $G$ is a simple path on $n$ vertices. In that special case, the graph associahedron of $G$ is the $n$-dimensional associahedron. When however, $G$ is the complete graph on $n$ vertices, then its graph associahedron is the $(n-1)$-dimensional permutohedron, the search trees on $G$ are one-to-one with permutations in $S_n$, and the rotation distance is simply the number of inversions between two permutations.

Here, we are interested in a different special case of graph associahedra: the \emph{split graph} associahedra. We recall the following standard definition.

\begin{defn}
    A graph $G=(V,E)$ is a \emph{split graph} if $V$ can be partitioned into two subsets $P$ and $Q$ inducing respectively a clique and an independent set in $G$.
    A split graph is \emph{complete} if it is maximal for this property, hence if every vertex of $P$ is adjacent with every vertex of $Q$.
\end{defn}

Throughout the paper, $G$ will be a fixed complete split graph, formed from two sets $P$ and $Q$ of vertices. We will denote by $p$ the number of vertices in $P$ and by $q$ the number of vertices in $Q$.
We introduce the following definition.
\begin{defn}
    A \emph{broom} $T$ is a rooted tree composed of a simple path which we shall refer to as the \emph{handle} of the broom, at one end of which all the leaves are attached. 
    The other end of the handle is the root of $T$. 
\end{defn}

A search tree on the complete split graph $G$ is always a broom~\cite{CardinalPourninValencia2022}:
The set of vertices of $G$ in the handle of $T$ is $P\cup{S}$ where $S$ is a subset of $Q$ and the leaves of $T$ are the vertices of $G$ contained in $Q\mathord{\setminus}S$. Note that by convention, in the case when $S$ is equal to $Q$, we consider that $T$ has no leaf. The vertex of the handle other than the root of $T$ (to which the leaves, if any, are attached) always belongs to $P$. Hereafter we will think of a broom as ordered from the root at the top to the leaves at the bottom and will say a vertex $u$ is \emph{above} a vertex $v$ when the distance to the root is smaller for $u$ than for $v$. The four trees $T_1$ to $T_4$ shown in Figure \ref{CPV2.sec.0.5.fig.1}, for instance, are brooms because they are search trees on a star. The search tree $T_1$ has no leaf and the vertex labeled $1$ in its handle is above the vertex labeled $3$. Moreover, $T_2$, $T_3$, and $T_4$ have $1$, $2$, and $3$ leaves, respectively.


Since the search trees on $G$ are brooms we will speak of brooms \emph{on} $G$. As mentioned above, the vertices of $\mathcal{A}(G)$ corresponds to these brooms and its edges to the rotations between two brooms. There are three types of rotations between brooms \cite{CardinalPourninValencia2022}. 
\begin{enumerate}
    \item The first type of rotations consists in exchanging two consecutive vertices $u$ and $v$ within the handle of the broom. If $v$ is the vertex attached to the leaves (which belongs to $P$), then this type of rotation requires that $u$ also belongs to $P$. 
    \item In the second type of rotation, $v$ is the vertex attached to the leaves and the vertex $u$ just above it belongs to $Q$. Then $u$ is removed from the handle, and inserted as a leaf, while $v$ is reattached to the vertex that was above $u$. In this case, the length of the handle decreases by $1$ and the number of leaves increases by $1$. 
    \item The third type of rotations is the inverse of the second type: a leaf $u$ of $T$ is removed and inserted in the handle of the broom between the vertex $v$ to which the leaves are attached and the vertex just above $v$. 
\end{enumerate}

These three types of rotations exactly describe the edges of $\mathcal{A}(G)$ in terms of brooms \cite{CardinalPourninValencia2022}.

\begin{rem}
\label{rem:partial}
In the introduction, we mentioned that, when $G$ is a star (that is, when $p$ is equal to $1$) whose leaves are labeled from $1$ to $n$, then the brooms on $G$ and their rotations can be replaced by the partial permutations of $\{1,\ldots,n\}$ where $n$ stands for the number of vertices of $G$ and three types of elementary operations. In fact, the partial permutation corresponding to a broom $T$ can be recovered by reading the elements of $Q$ contained in the handle of $T$ starting from the top. Then, the transpositions in the partial permutations correspond to exchanging two vertices in the handle of $T$. Likewise, the addition at the end of the partial permutation of an element of $\{1,\ldots,n\}$ correspond to the rotation that moves a leaf of $T$ into its handle, and the removal of the last element of the partial permutation to the inverse rotation. When $G$ is the claw shown in Figure \ref{CPV2.sec.0.5.fig.0}, the stellohedron $\mathcal{A}(G)$ is depicted in Figure~\ref{CPV2.sec.0.5.fig.1} where its vertices are labeled using the partial permutations of $\{1,2,3\}$ (and the empty partial permutation is denoted by $\emptyset$).
\end{rem}

Let us conclude this section by observing that the face complex of $\mathcal{A}(G)$ contains a number of copies of the permutohedra of each dimension between $p-1$ and $p+q-2$. In particular, the subgraph induced in the $1$-skeleton of $\mathcal{A}(G)$ by the brooms whose handle contains a given subset $S$ of $k$ vertices of $Q$, and whose bottommost vertex is a given element $v$ of $P$ is a permutohedron of dimension $p+k-2$. Its edges correspond to all the possible transpositions of consecutive vertices that do not involve $v$ along the handles of these brooms or, in the terminology of permutations, to all the transpositions of consecutive entries in a permutation of $(P\mathord{\setminus}\{v\})\cup{S}$. For instance, when $G$ is the claw from Figure \ref{CPV2.sec.0.5.fig.0} and $S=\{1,2,3\}$, one recovers the $2$-dimensional permutohedron shown as the outer hexagon on the left of Figure \ref{CPV2.sec.0.5.fig.1}. Note that, in this case, $v$ is necessarily the center of the star as $P=\{v\}$. Similarly, when $S=\{1,2\}$, we get the $1$-dimensional permutohedron corresponding to the edge between the vertices $12$ and $21$ on the left of Figure \ref{CPV2.sec.0.5.fig.1}.

We therefore immediately obtain the following statement.
Recall that $G$ is a complete split graph with a vertex bipartition $P\cup Q$ such that $|P|=p$ and $|Q|=q$. 
\begin{prop}\label{CPV2.sec.0.5.prop.1}
Consider an integer $k$ such that $1\leq{k}\leq{q}$. 
The face complex of the graph associahedron $\mathcal{A}(G)$ contains at least
$p{q\choose{k}}$
different copies of the $(p+k-2)$-dimensional permutohedron.
\end{prop}

In particular, Proposition \ref{CPV2.sec.0.5.prop.1} singles out the $p$ permutohedral facets of $\mathcal{A}(G)$ obtained by taking the whole of $Q$ for $S$. Each of these permutohedral facets corresponds to a choice for the vertex $v$ within $P$.

\section{A quadratic model for rotation distance}\label{CPV2.sec.1}

Throughout the section and the next two, $T_1$ and $T_2$ are two fixed brooms on the complete split graph $G$. We aim at establishing an expression for their rotation distance or, equivalently for the distance of the vertices they correspond to in the graph of $\mathcal{A}(G)$. We denote by $Y$ the set of the vertices of $Q$ that belong to the handle of both $T_1$ and $T_2$. We also denote by $\sigma$ the permutation of $\{1,\ldots,p\}$ obtained by relabeling the vertices in $P$ from $1$ to $p$ according to their order from top to bottom within the handle of $T_1$ and then reading these labels again from top to bottom within the handle of $T_2$. The number of inversions $\mathrm{inv}(\sigma)$ of $\sigma$ will appear in the expression of the rotation distance between $T_1$ and $T_2$. Indeed, if $q$ is equal to $0$, then $\mathcal{A}(G)$ is the permutohedron and the distance between $T_1$ and $T_2$ is precisely equal to $\mathrm{inv}(\sigma)$.

Consider a vertex $u$ in $Q\mathord{\setminus}Y$. According to the definition of $Y$, the handles of $T_1$ and $T_2$ cannot both contain $u$. If $u$ belongs to one of these handles, we denote by $A_u$ the set of the vertices in $P$ that lie below $u$ in that handle. If however, $u$ is a leaf in both brooms, then $A_u$ is the empty set.

Now consider a vertex $u$ of $Y$. We distinguish two subsets of $P$:

\begin{itemize}
\item[(i)] the set $B_u$ of all vertices in $P$ that are above $u$ in one of the two brooms $T_1$ or $T_2$ and below $u$ in the other and
\item[(ii)] the set $C_u$ of all vertices in $P$ that are below $u$ in both of the brooms $T_1$ and $T_2$.
\end{itemize}

We also distinguish three subsets of $Q$ (still for a vertex $u$ in $Y$):

\begin{itemize}
\item[(iii)] the set $D_u$ of all vertices in $Q\mathord{\setminus}Y$ that are above $u$ in one of the two brooms $T_1$ or $T_2$,
\item[(iv)] the set $E_u$ of all vertices in $Y$ that are above $u$ in one of the two brooms $T_1$ or $T_2$ and below $u$ in the other, and
\item[(v)] the set $F_u$ of all vertices in $Y$ that are below $u$ in both of the brooms $T_1$ and $T_2$.
\end{itemize}

Using this notations, we can prove the following lower bound.

\begin{lem}\label{CPV2.sec.1.lem.1}
Consider a sequence of rotations that transform $T_1$ into $T_2$. The number of rotations in that sequence is at least
\begin{equation}\label{CPV2.sec.1.lem.1.eq.1}
\begin{array}{l}
\displaystyle\mathrm{inv}(\sigma)+\!\!\!\sum_{u\in{Q\mathord{\setminus}Y}}\!\!\!|A_u|+\sum_{u\in{Y}}(|B_u|+2|C_u|x_u+|D_u|(1-x_u))\hspace{0.15\textwidth}\\[\bigskipamount]
\displaystyle\hfill+\frac{1}{2}\sum_{u\in{Y}}\sum_{v\in{E_u}}(1+x_u)(1-x_v)+2\sum_{u\in{Y}}\sum_{v\in{F_u}}x_u(1-x_v)
\end{array}
\end{equation}
where, for every vertex $u$ in $Y$,
$$
x_u=\left\{
\begin{array}{l}
1\mbox{ if $u$ is a leaf of some broom along the sequence of rotations,}\\
0\mbox{ if $u$ remains in the handle of all brooms in that sequence.}\\
\end{array}
\right.
$$
\end{lem}
\begin{proof}
Consider two vertices $u$ and $v$ of $G$. We will bound the number of rotations exchanging $u$ and $v$. If both $u$ and $v$ belong to $P$ and they do not have the same order in the handle of $T_1$ and $T_2$, then there is at least one rotation that exchanges them because the vertices of $P$ must belong to the handle of all brooms. Hence, the number of the rotations involving two vertices from $P$ is at least the number of inversions of $\sigma$, as desired.

If $u$ is a vertex from $Q\mathord{\setminus}Y$ and $v$ is a vertex from $A_u$, then at least one rotation involves $u$ and $v$ along the considered path, because $u$ is a leaf in one of the brooms $T_1$ or $T_2$ and is above $v$ in the other while $v$ must remain in the handle of all brooms along the sequence. Therefore, at least
\begin{equation}\label{CPV2.sec.1.lem.1.eq.7}
\sum_{u\in{Q\mathord{\setminus}Y}}\!\!|A_u|
\end{equation}
rotations involve a vertex from $Q\mathord{\setminus}Y$ and a vertex from $P$. 

Now suppose that $u$ belongs to $Y$. If $v$ belongs to $B_u$ then it has to be exchanged with $u$ by at least one rotation. Indeed, $v$ lies above $u$ in one of the brooms and below it in the other. Moreover, $v$ must remain in the handle of all brooms along the sequence of rotations because it belongs to $P$. If $v$ belongs to $C_u$ and $u$ becomes a leaf of some broom along the considered sequence of rotations, then it has to be exchanged with $v$ by two rotations (because it lies above $v$ in both brooms and $v$ must remain in the handle). Hence at least
\begin{equation}\label{CPV2.sec.1.lem.1.eq.2}
\sum_{u\in{Y}}|B_u|+2|C_u|x_u
\end{equation}
of the considered rotations involve a vertex in $Y$ and a vertex from $P$.
 
If $v$ belongs to $D_u$ and $u$ remains in the handle of all brooms along the sequence of rotations, then it has to be exchanged with $v$ by a rotation (because $v$ lies above $u$ in one of the brooms and is a leaf of the other). Hence, the number of rotations involving a vertex of $Y$ and a vertex from $Q\mathord{\setminus}Y$ is at least
\begin{equation}\label{CPV2.sec.1.lem.1.eq.3}
\sum_{u\in{Y}}|D_u|(1-x_u)\mbox{.}
\end{equation}

Now suppose that both $u$ and $v$ belong to $Y$. If $v$ belongs to $E_u$ and $u$ becomes a leaf of some broom along the sequence of rotations but $v$ does not, then $u$ and $v$ must be exchanged by at least one rotation. If $v$ belongs to $E_u$ but both $u$ and $v$ remain in the handle of all brooms along the sequence of rotations, then at least one rotation must involve $u$ and $v$ because the order of these vertices is not the same within the handle of $T_1$ and the handle of $T_2$. Hence, the number of rotations involving $u$ and a vertex from $E_u$ is at least
\begin{equation}\label{CPV2.sec.1.lem.1.eq.4}
\sum_{v\in{E_u}}x_u(1-x_v)+\sum_{v\in{E_u}}(1-x_u)(1-x_v)\mbox{.}
\end{equation}

Note that this expression could be simplified but it will be useful later to keep it in this form. If $v$ belongs to $F_u$ and $u$ becomes a leaf of some broom along the sequence of rotations but $v$ does not, then $u$ and $v$ must be exchanged by at least two rotations (one when $u$ goes down to become a leaf and another when it goes back up along the handle of the broom). The number of rotations involving $u$ and a vertex from $F_u$ is therefore at least
\begin{equation}\label{CPV2.sec.1.lem.1.eq.5}
2\sum_{v\in{F_u}}x_u(1-x_v)\mbox{.}
\end{equation}

Observe that if one sums (\ref{CPV2.sec.1.lem.1.eq.5}) when $u$ ranges over $Y$, no rotation is counted twice as the terms $x_u(1-x_v)$ and $x_v(1-x_u)$ cannot both be non-zero. The same goes for the first sum from (\ref{CPV2.sec.1.lem.1.eq.4}). However, if one sums to second sum from (\ref{CPV2.sec.1.lem.1.eq.4}) when $u$ ranges over $Y$ each of the corresponding rotations gets counted twice. As a consequence, there must be at least
\begin{equation}\label{CPV2.sec.1.lem.1.eq.6}
\frac{1}{2}\sum_{u\in{Y}}\sum_{v\in{E_u}}(1+x_u)(1-x_v)+2\sum_{u\in{Y}}\sum_{v\in{F_u}}x_u(1-x_v)
\end{equation}
rotations involving two vertices from $Y$. Summing $\mathrm{inv}(\sigma)$ with (\ref{CPV2.sec.1.lem.1.eq.7}), (\ref{CPV2.sec.1.lem.1.eq.2}), (\ref{CPV2.sec.1.lem.1.eq.3}), and (\ref{CPV2.sec.1.lem.1.eq.6}) provides the announced lower bound.  
\end{proof}

Let us denote the coordinates of a point $x$ from $\{0,1\}^{|Y|}$ by $x_u$ where $u$ ranges over $Y$. With that notation, (\ref{CPV2.sec.1.lem.1.eq.1}) can be thought of as a function $f:\{0,1\}^{|Y|}\rightarrow\mathbb{N}$ of the variable $x$. The lower bound stated by Lemma \ref{CPV2.sec.1.lem.1} turns out to be sharp. In order to see that, we will construct, for each point $x$ from $\{0,1\}^{|Y|}$, a sequence of exactly $f(x)$ rotations that transform $T_1$ into $T_2$. In that sequence, each vertex $u$ from $Y$ will appear as a leaf of some broom precisely when the coordinate of $x$ it corresponds to is equal to $1$.

Informally, the sequence of rotations that we are going to build first moves within $T_1$ those vertices of $Q$ contained in the handle that are not in $Y$ or whose corresponding coordinate of $x$ is equal to $1$. These vertices are moved down to the leaves, which can be done without using any  rotation that involves two of them. Then, the vertices remaining in the handle of the broom are permuted within that handle into their order in the handle of $T_2$, resulting in a broom $T$. This portion of the sequence of rotations remains in one or several of the permutohedral faces of $\mathcal{A}(G)$ but it should be noted that it may not remain entirely in a single permutohedral face of $\mathcal{A}(G)$ because the bottommost vertex of the handle can change along the sequence (though this vertex always belongs to $P$). Finally, the leaves of $T$ that appear in the handle of $T_2$ are moved up, again without using any rotation that involves two of them.

This construction is illustrated in Figure \ref{CPV2.sec.1.fig.1} in the case when $G$ is the claw from Figure \ref{CPV2.sec.0.5.fig.0}. In that example, the three vertices in $Q$ belong to the handle of both $T_1$ and $T_2$ hence $Y=\{1,2,3\}$. Moreover, we set $x_1$ to $1$ and the other two coordinates of $x$ to $0$ so the only vertex from $Y$ that becomes a leaf along the considered path is $1$. 
\begin{figure}
\begin{centering}
\includegraphics[scale=1]{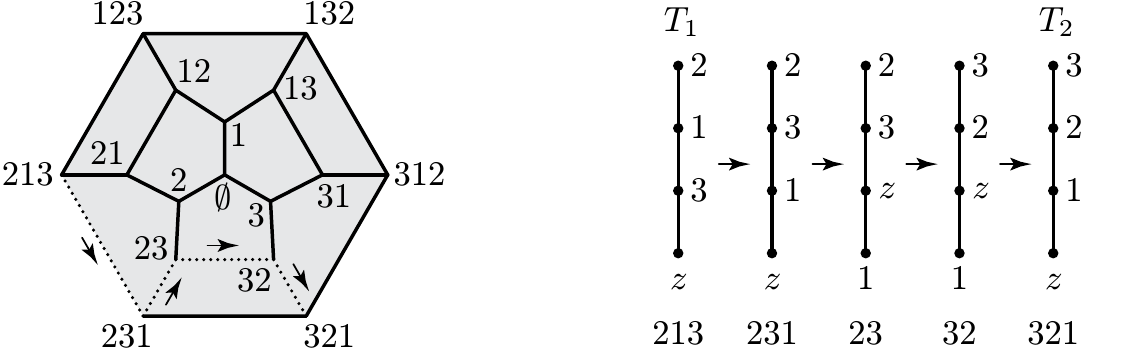}
\caption{The path built in the proof of Lemma \ref{CPV2.sec.1.lem.2} in the case of two brooms $T_1$ and $T_2$ on the claw $G$ shown in Figure \ref{CPV2.sec.0.5.fig.0}. The sequence of brooms is shown on the right, and the corresponding path in $\mathcal{A}(G)$ is dotted on the left.}\label{CPV2.sec.1.fig.1}
\end{centering}
\end{figure}
Note that this does not produce a geodesic path between $T_1$ and $T_2$ but, by Lemma \ref{CPV2.sec.1.lem.1}, this is a shortest path under the constraint that $1$ becomes a leaf along the path while $2$ and $3$ do not.

\begin{lem}\label{CPV2.sec.1.lem.2}
For any point $x$ in $\{0,1\}^{|Y|}$, there exists a sequence of exactly $f(x)$ rotations that transform $T_1$ into $T_2$ such that a vertex $u$ in $Y$ appears as a leaf of some broom along the sequence of rotations if and only if $x_u=1$.
\end{lem}
\begin{proof}
Consider the subset $S$ of $Y$ of the vertices $u$ such that $x_u=1$. Further denote by $S_1$ the set of the vertices in $P\cup{Q}$ that belong to the handle of $T_1$ but are leaves of $T_2$. Likewise, let $S_2$ be the set of the vertices in $P\cup{Q}$ in the handle of $T_2$ that are leaves of $T_1$. We will explicitly build a sequence of rotations between $T_1$ and $T_2$. First consider the vertex from $S\cup{S_1}$ that is below all the other vertices in $S\cup{S_1}$ within the handle of $T_1$ and perform a sequence of rotation in order to move it down until it becomes a leaf. Do the same for the second lowest vertex of $S\cup{S_1}$ in the handle of $T_1$ and so on until all the vertices in $S\cup{S_1}$ have become leaves. Denote by $T'_1$ the broom that results from that sequence of rotations and note that on the way from $T_1$ to $T'_1$, no rotation has involved two vertices from $S\cup{S_1}$. Do the same starting from $T_2$ in order to obtain a broom $T'_2$ where all the vertices in $S\cup{S_2}$ have become leaves but no two of them were involved in the same rotation. The total number of rotations that have involved a vertex from $S$ so far is
\begin{equation}\label{CPV2.sec.1.lem.2.eq.1}
\sum_{u\in{S}}(|B_u|+2|C_u|+|E_u\mathord{\setminus}S|+2|F_u\mathord{\setminus}S|)\mbox{.}
\end{equation}

Indeed, each vertex in $S$ has been involved in as many rotations as there are vertices in the handles of $T_1$ and $T_2$ below it that belong to either $P$ or $Y\mathord{\setminus}S$. More precisely, each vertex $u$ in $S$ has been involved in one rotation with each vertex in $B_u$ and with each vertex in $E_u\mathord{\setminus}{S}$. It has also been involved in two rotations with each vertex in $C_u$ and with each vertex in $F_u\mathord{\setminus}{S}$. 

The number of rotations that have involved a vertex from $Q\mathord{\setminus}Y$ so far is
\begin{equation}\label{CPV2.sec.1.lem.2.eq.2}
\sum_{u\in{Q\mathord{\setminus}Y}}\!\!|A_u|+\!\!\sum_{u\in{Y\mathord{\setminus}S}}\!\!|D_u|\mbox{.}
\end{equation}

Indeed, each of the vertex from $Q\mathord{\setminus}Y$ has been involved in exactly one rotation with each of the vertices from $P$ below it (note that if a vertex in $Q\mathord{\setminus}Y$ is a leaf of both $T_1$ and $T_2$, then $A_u$ is empty). Moreover, each vertex contained in $Y\mathord{\setminus}S$ has been involved in exactly one rotation with each of the vertices from $Q\mathord{\setminus}Y$ above it in either of the brooms $T_1$ and $T_2$.

As $S$ consists of the vertices in $Y$ such that $x_u$ is equal to $1$ and $Y\mathord{\setminus}S$ of the vertices in $Y$ such that $x_u$ is equal to $0$,
$$
\left\{
\begin{array}{l}
\displaystyle|E_u\mathord{\setminus}S|=\sum_{v\in{E_u}}(1-x_v)\mbox{,}\\[1.5\bigskipamount]
\displaystyle|F_u\mathord{\setminus}S|=\sum_{v\in{F_u}}(1-x_v)\mbox{.}\\[-1.5\bigskipamount]
\end{array}
\right.
\vspace{\smallskipamount}
$$

The sum (\ref{CPV2.sec.1.lem.2.eq.1}) can be rewritten as
\begin{equation}\label{CPV2.sec.1.lem.2.eq.3}
\sum_{u\in{Y}}\left[(|B_u|+2|C_u|)x_u+\sum_{v\in{E_u}}x_u(1-x_v)+2\sum_{v\in{F_u}}x_u(1-x_v)\right]
\end{equation}
and the sum (\ref{CPV2.sec.1.lem.2.eq.2}) as
\begin{equation}\label{CPV2.sec.1.lem.2.eq.4}
\sum_{u\in{Q\mathord{\setminus}Y}}\!\!|A_u|+\!\!\sum_{u\in{Y}}|D_u|(1-x_u)\mbox{.}
\end{equation}

The number of rotations used to transform $T_1$ into $T'_1$ and $T_2$ into $T'_2$ is therefore the sum of (\ref{CPV2.sec.1.lem.2.eq.3}) and (\ref{CPV2.sec.1.lem.2.eq.4}). There remains to transform $T'_1$ into $T'_2$. Note that the vertices in the handles of these brooms are the same: they are precisely the vertices contained in $P\cup(Y\mathord{\setminus}S)$. Therefore, the distance between $T'_1$ into $T'_2$ is actually the number of inversions of the permutation $\tau$ obtained by relabeling the vertices of $P\cup(Y\mathord{\setminus}S)$ from $1$ to $|P\cup(Y\mathord{\setminus}S)|$ in their order from the top of the handle of $T'_1$ to its bottom and their reading these labels again in the handle of $T'_2$ from the top of the handle to its bottom. By construction, the number of inversions of $\tau$ that involve two vertices in $P$ is $\mathrm{inv}(\sigma)$. The number of inversions of $\tau$ that involve a vertex in $Y\mathord{\setminus}S$ and a vertex in $P$ is
\begin{equation}\label{CPV2.sec.1.lem.2.eq.5}
\sum_{u\in{Y\mathord{\setminus}S}}\!\!|B_u|
\end{equation}
and the number of inversions that involve two vertices from $Y\mathord{\setminus}S$ is
\begin{equation}\label{CPV2.sec.1.lem.2.eq.6}
\frac{1}{2}\sum_{u\in{Y\mathord{\setminus}S}}\sum_{v\in{E_u}}(1-x_v)\mbox{.}
\end{equation}

Note that there is a $1/2$ coefficient in the last expression because otherwise each inversion between two vertices $u$ and $v$ in $Y\mathord{\setminus}S$ would be counted twice: once as the term $1-x_v$ in the sum over $E_u$ and once as the term $1-x_u$ in the sum over $E_v$. Recall in particular that $v$ belongs to $E_u$ if and only if $u$ belongs to $E_v$. Now observe that since $x_u$ is equal to $0$ when $u$ belongs to $Y\mathord{\setminus}S$, the sum of (\ref{CPV2.sec.1.lem.2.eq.5}) and (\ref{CPV2.sec.1.lem.2.eq.6}) can be rewritten as
\begin{equation}\label{CPV2.sec.1.lem.2.eq.7}
\sum_{u\in{Y}}\!\!|B_u|(1-x_u)+\frac{1}{2}\sum_{u\in{Y}}\sum_{v\in{E_u}}(1-x_u)(1-x_v)\mbox{.}
\end{equation}

Summing $\mathrm{inv}(\sigma)$ with (\ref{CPV2.sec.1.lem.2.eq.3}), (\ref{CPV2.sec.1.lem.2.eq.4}), and (\ref{CPV2.sec.1.lem.2.eq.7}) shows that the length of the constructed path is indeed $f(x)$ as desired.
\end{proof}

\section{Rotation distance and quadratic $0/1$ optimization}\label{CPV2.sec.2}

In view of Lemmas \ref{CPV2.sec.1.lem.1} and \ref{CPV2.sec.1.lem.2}, finding the rotation distance between $T_1$ and $T_2$ amounts to find the minimum of $f$ when $x$ ranges over $\{0,1\}^{|Y|}$. Note that $f$ is a quadratic function of $x$. Minimizing a quadratic function of $0/1$ variables is a well-known problem of combinatorial optimization. This problem is \NP\nobreakdash-hard in general as it admits the weighted max-cut problem as a special case \cite{Ivanescu1965}. However, depending on the form of the quadratic function to minimize, it can be polynomial time solvable \cite{AllemandFukudaLieblingSteiner2001,CelaPunnen2022,PicardRatliff1975}. This is the case in particular when the coefficients of all the products between two distinct variables in the expression of the quadratic function are non-positive \cite{CelaPunnen2022,PicardRatliff1975}. One can see by looking at (\ref{CPV2.sec.1.lem.1.eq.1}) that $f$ is precisely of this form. The following proposition explicitly provides the constant, linear, and quadratic terms in the expression of $f$.

\begin{prop}\label{CPV2.sec.1.prop.1}
$f(x)=c+\ell (x)+q(x)$ where
$$
\left\{
\begin{array}{l}
\displaystyle{c=\mathrm{inv}(\sigma)+\!\!\!\sum_{u\in{Q\mathord{\setminus}Y}}\!\!\!|A_u|+\sum_{u\in{Y}}\left(|B_u|+|D_u|+\frac{1}{2}|E_u|\right)\mbox{,}}\\[1.5\bigskipamount]
\displaystyle{\ell (x)=\sum_{u\in{Y}}(2|C_u|-|D_u|+2|F_u|)x_u}\mbox{,}\\[1.5\bigskipamount]
\displaystyle{q(x)=\frac 12 \sum_{u\in{Y}}\sum_{v\in{Y}}q_{u,v}x_ux_v}\mbox{,}\\
\end{array}
\right.
$$
and for every two vertices $u$ and $v$ from $Y$,
$$
q_{u,v}=
\left\{
\begin{array}{rl}
0 & \mbox{ when }u=v\mbox{,}\\
-1 & \mbox{ when }v\in{E_u}\mbox{, and}\\
-2 & \mbox{ otherwise}\mbox{.}\\
\end{array}
\right.
$$
\end{prop}

\begin{proof}
Observe that
\begin{equation}\label{CPV2.sec.1.prop.1.eq.1}
\begin{array}{l}
\displaystyle\frac{1}{2}\sum_{u\in{Y}}\sum_{v\in{E_u}}(1+x_u)(1-x_v)=\frac{1}{2}\sum_{u\in{Y}}|E_u|(1+x_u)-\frac{1}{2}\sum_{u\in{Y}}\sum_{v\in{E_u}}x_v\\[1.5\bigskipamount]
\displaystyle\hfill-\frac{1}{2}\sum_{u\in{Y}}\sum_{v\in{E_u}}x_ux_v\mbox{.}
\end{array}
\end{equation}

However, recall that the vertex $v$ belongs to $E_u$ if and only if the vertex $u$ belongs to $E_v$. As a consequence, the number of vertices $u$ in $Y$ such that $v$ belongs to $E_u$ is equal to $|E_v|$, and we obtain
\begin{equation}\label{CPV2.sec.1.prop.1.eq.2}
\sum_{u\in{Y}}\sum_{v\in{E_u}}x_v=\sum_{v\in{Y}}|E_v|x_v\mbox{.}
\end{equation}

Combining (\ref{CPV2.sec.1.prop.1.eq.1}) and (\ref{CPV2.sec.1.prop.1.eq.2}) yields
$$
\frac{1}{2}\sum_{u\in{Y}}\sum_{v\in{E_u}}(1-x_u)(1-x_v)=\frac{1}{2}\sum_{u\in{Y}}|E_u|-\frac{1}{2}\sum_{u\in{Y}}\sum_{v\in{E_u}}x_ux_v
$$

Further observe that
$$
2\sum_{u\in{Y}}\sum_{v\in{F_u}}x_u(1-x_v)=2\sum_{u\in{Y}}|F_u|x_u-2\sum_{u\in{Y}}\sum_{v\in{F_u}}x_ux_v\mbox{.}
$$

As a consequence,
\begin{equation}\label{CPV2.sec.1.prop.1.eq.3}
f(x)=c+\ell (x)-\sum_{u\in{Y}}\!\left[\frac{1}{2}\sum_{v\in{E_u}}x_ux_v+2\sum_{v\in{F_u}}x_ux_v\right]\!\!.
\end{equation}

It remains to compute the coefficient of each term of the form $x_ux_v$ in the right-hand side of (\ref{CPV2.sec.1.prop.1.eq.3}), where $u$ and $v$ are two vertices from $Y$. First note that there is no such term when $u=v$ because $u$ does not belong to $E_u$ or to $F_u$. Let $u$ and $v$ be two distinct vertices from $Y$. Recall that $v$ belongs to $E_u$ if and only if $u$ belongs to $E_v$. Moreover, $E_u$ is disjoint from $F_u$ and $E_v$ from $F_v$. Hence, if $v$ belongs to $E_u$, then there are exactly two terms involving both $x_u$ and $x_v$ in the right-hand side of (\ref{CPV2.sec.1.prop.1.eq.3}), one in the sum over the elements of $E_u$ and one in the sum over the elements of $E_v$ whose sum is $-x_ux_v$. Now if $v$ belongs to $F_u$, then $u$ cannot belong to $F_v$ because $v$ is below $u$ in both $T_1$ and $T_2$. Hence the term $-2x_ux_v$ (that comes from the sum over the elements of $F_u$) is the only one that involves both $x_u$ and $x_v$ in the right-hand side of (\ref{CPV2.sec.1.prop.1.eq.1}). If however, $v$ does not belong to $E_u\cup{F_u}$ then $u$ must belong to $F_v$. Indeed, as $v$ is not in $E_u$, it must be below $u$ in both $T_1$ and $T_2$ or above $u$ in both $T_1$ and $T_2$. However, the former is impossible because $v$ is not in $F_u$. Therefore, the term $-2x_vx_u$ (that comes from the sum over the elements of $F_v$) is the only one that involves both $x_u$ and $x_v$ in the left-hand side of (\ref{CPV2.sec.1.prop.1.eq.1}). As a consequence,
$$
-\sum_{u\in{Y}}\!\left[\frac{1}{2}\sum_{v\in{E_u}}x_ux_v+2\sum_{v\in{F_u}}x_ux_v\right]\!\!=q(x)
$$
which completes the proof of the proposition.
\end{proof}

According to Lemmas \ref{CPV2.sec.1.lem.1} and \ref{CPV2.sec.1.lem.2}, computing the distance between $T_1$ and $T_2$ amounts to minimize $f$, a quadratic function of $x$, over $\{0,1\}^{|Y|}$. By Proposition~\ref{CPV2.sec.1.prop.1}, all the terms of the form $x_ux_v$ in the expression of $f$ have non-positive coefficients. As a consequence, this $0/1$ quadratic minimization problem is polynomial time solvable in $|Y|$ \cite{CelaPunnen2022,PicardRatliff1975}, and we obtain the following.

\begin{thm}\label{CPV2.sec.2.thm.1}
The rotation distance of two given brooms on a complete split graph with $n$ vertices can be computed in time polynomial in $n$.
\end{thm}

Not only is the minimum of $f$ over $\{0,1\}^{|Y|}$ equal to the rotation distance between $T_1$ and $T_2$, but any value of $x$ at which $f$ reaches its minimum provides a geodesic path between $T_1$ and $T_2$ via Lemma \ref{CPV2.sec.1.lem.2}. Let us illustrate this in the case of the stellohedron, hence when $P$ is a singleton. We consider two brooms $T_1$ and $T_2$ such that all the vertices in $Q$ belong to the handle of both brooms but their order from top to bottom is opposite in $T_1$ and in $T_2$. The unique vertex in $P$ appears at the bottom of both handles.

These particular pairs of brooms are interesting for two reasons. The first is that their rotation distance is equal to the diameter of the graph of $\mathcal{A}(G)$ \cite{MannevillePilaud2015}. The second is that, even though the unique permutohedral facet of $\mathcal{A}(G)$ admits both $T_1$ and $T_2$ as vertices (as discussed in the end of Section \ref{CPV2.sec.0.5}), none of the geodesic paths between $T_1$ and $T_2$ in the graph of $\mathcal{A}(G)$ remain entirely in that facet when $q$ is greater than $5$ \cite{MannevillePilaud2015}. In other words, within the graph of $\mathcal{A}(G)$, the subgraph consisting of the vertices and edges of that permutohedral facet is sometimes not strongly convex. This particular property has been studied under different names and in different contexts related to flip graphs (see, for instance \cite{DisarloParlier2019,MannevillePilaud2015,PourninWang2021,SleatorTarjanThurston1988}).

For the brooms $T_1$ and $T_2$ defined above, $Y$ is equal to $Q$. Moreover, for every vertex $u$ in $Y$, the sets $B_u$, $D_u$ and $F_u$ are empty while $C_u$ is equal to $P$ and $E_u$ to $Y\mathord{\setminus}\{u\}$. Hence, it follows from Proposition \ref{CPV2.sec.1.prop.1} that
$$
f(x)=\frac{q(q-1)}{2}+2\sum_{u\in{Y}}x_u-\frac{1}{2}\sum_{u\in{Y}}\sum_{\substack{v\in{Y}\\v\neq{u}}}x_ux_v\mbox{.}
$$

However, as $x_u^2$ is equal to $x_u$ for all the vertices $u$ in $Y$,
$$
f(x)=\frac{q(q-1)}{2}+\frac{5}{2}\sum_{u\in{Y}}x_u-\frac{1}{2}\!\left(\sum_{u\in{Y}}x_u\right)^{\!2}\mbox{.}
$$

Minimizing $f$ then amounts to search for the number $i$ of vertices from $Q$ that should move down and become a leaf along a geodesic path. As
\begin{equation}\label{CPV2.sec.2.eq.1}
\frac{q(q-1)}{2}+\frac{5}{2}i-\frac{1}{2}i^2
\end{equation}
is a concave function of $i$, in order to minimize this quantity, it suffices to compute it when $i$ is equal to $0$ and when $i$ is equal to $q$. In the former case, (\ref{CPV2.sec.2.eq.1}) is equal to $q(q-1)/2$ and in the latter, to $2q$. Therefore, we recover the observation from \cite{MannevillePilaud2015}: if $q$ is greater than $5$, then
$$
\frac{q(q-1)}{2}>2q
$$
and all the vertices in $Q$ have to become leaves along any geodesic path between $T_1$ and $T_2$. 
In polyhedral terms, all the geodesic paths between $T_1$ and $T_2$ must leave the permutohedral facet of $\mathcal{A}(G)$. 

Let us now turn our attention to the form of the solutions to our quadratic $0/1$ minimization problem. It might seem counter intuitive that, given a vertex $u$ in $Y$ and a vertex $v$ in $F_u$, the former vertex should be allowed to move down to the leaves of the broom (that is, $x_u$ is set to $1$), while the latter remains in the handle ($x_v$ is set to $0$). Indeed, if it is shorter to move $u$ to the leaves of the broom, then it should also be shorter to do the same with $v$.

The following lemma formalizes this observation.

\begin{lem}\label{CPV2.sec.2.lem.1}
The function $f$ reaches its minimum over $\{0,1\}^{|Y|}$ at a point $x$ such that $x_u\leq{x_v}$ for any vertex $u$ in $Y$ and any vertex $v$ in $F_u$.
\end{lem}
\begin{proof}
Consider a point $x$ in $\{0,1\}^{|Y|}$ and denote by $m(x)$ the number of pairs of vertices $u$ and $v$ of $G$ such that $u$ belongs to $Y$ and $v$ to $F_u$ but $x_u$ is equal to $1$ and $x_v$ to $0$. Let us assume that $f$ is minimal at $x$ and that, under this requirement on $x$, $m(x)$ is minimal. If $m(x)$ is equal to $0$ then we are done. Otherwise, let us aim for a contradiction. As $m(x)$ is positive then there exist a vertex $u$ in $Y$ and a vertex $v$ in $F_u$ such that $x_u$ is equal to $1$ and $x_v$ to $0$. Consider two such vertices $u$ and $v$. We will require without loss of generality that $u$ is chosen among the possible candidates in such a way that $F_u$ is as large as possible. Likewise, $v$ is chosen within $F_u$ such that $F_v$ is the largest possible under the requirement that $x_v$ is equal to $0$.

Now consider the point $x'$ of $\{0,1\}^{|Y|}$ obtained from $x$ by changing $x_u$ to $0$. Similarly, let $x''$ be the point in $\{0,1\}^{|Y|}$ one gets by switching $x_v$ to $1$ in $x$. Note that $m(x')$ and $m(x'')$ are both less than $m(x)$ by our choice for $u$ and $v$. It follows from Proposition \ref{CPV2.sec.1.prop.1} that
$$
f(x)-f(x')=2|C_u|-|D_u|+2|F_u|+\sum_{w\in{Y}}q_{u,w}x_w
$$
and
$$
f(x)-f(x'')=-2|C_v|+|D_v|-2|F_v|-\sum_{w\in{Y}}q_{v,w}x_w\mbox{.}
$$

As $v$ belongs to $F_u$, that vertex lies below $u$ both in the handles of $T_1$ and in the handle of $T_2$. Hence, $C_v$ is a subset of $C_u$ and $D_u$ is a subset of $D_v$. Moreover, $F_v$ is a subset of $F_u$. As a consequence,
\begin{equation}\label{CPV2.sec.2.lem.1.eq.1}
2f(x)-f(x')-f(x'')\geq2|F_u\mathord{\setminus}F_v|+\sum_{w\in{Y}}(q_{u,w}-q_{v,w})x_w
\end{equation}

The right-hand side of this inequality must be non-negative. Indeed, consider a vertex $w$ in $Y$. According to Proposition \ref{CPV2.sec.1.prop.1}, the difference $q_{u,w}-q_{v,w}$ is only possibly negative when $w$ is equal to $v$ or when $w$ belongs to $F_u\cap{E_v}$. Moreover, in the latter case this difference is equal to $-1$. However, if $w$ is equal to $v$, then $x_w$ is equal to $0$ and the product $(q_{u,w}-q_{v,w})x_w$ vanishes. Therefore,
$$
2f(x)-f(x')-f(x'')\geq2|F_u\mathord{\setminus}F_v|-|F_u\cap{E_v}|\mbox{.}
$$

Now observe that $F_u\cap{E_v}$ is a subset of $F_u\mathord{\setminus}F_v$ because $E_v$ and $F_v$ are disjoint. As a consequence, $2f(x)-f(x')-f(x'')$ is non-negative. It follows that the differences $f(x)-f(x')$ and $f(x)-f(x'')$ cannot both be negative and that $f$ is minimal at $x'$ or at $x''$. However, by our choice for $u$ and $v$, we know that $m(x')$ and $m(x'')$ are both less than $m(x)$, contradicting the minimality of $m(x)$ among the points from $\{0,1\}^{|Y|}$ where $f$ reaches its minimum.
\end{proof}

\begin{rem}
If $x$ is a point from $\{0,1\}^{|Y|}$ such that $f$ reaches its minimum at $x$ and $x_u\leq{x_v}$ for any vertex $u$ in $Y$ and any vertex $v$ in $F_u$, then the last nested sum in (\ref{CPV2.sec.1.lem.1.eq.1}) vanishes at that point because $x_u$ cannot be equal to $1$ while $x_v$ is equal to $0$. Equivalently, the terms of the form $-2x_ux_v$ in the expression of $q$ from Proposition \ref{CPV2.sec.1.prop.1} and the terms of the form $2|F_u|x_u$ in the expression of $\ell$ cancel because either $x_u$ is equal to $0$ or both $x_u$ and $x_v$ are equal to $1$. However, it is important to note that, even though Lemma \ref{CPV2.sec.2.lem.1} states that such a minimum for $f$ always exists, the last nested sum in (\ref{CPV2.sec.1.lem.1.eq.1}) is still required in order for $f$ to properly model the rotation distance between $T_1$ and $T_2$ and for the optimization to yield a correct optimum. 
\end{rem}

\section{Rotation distance and the minimum cut problem}\label{CPV2.sec.3}

In this section, we show how computing the rotation distance between $T_1$ and $T_2$ can be reduced to an instance of the minimum cut problem in a weighted graph $H$. Our goal is to improve Theorem \ref{CPV2.sec.2.thm.1} into Theorem \ref{CPV2.sec.0.thm.1}.

We apply the techniques from \cite{CelaPunnen2022,PicardQueyranne1982,PicardRatliff1975}. The vertex set of $H$ is $Y\cup\{s,t\}$ where $Y$ is the subset of $Q$ defined in Section \ref{CPV2.sec.1} while $s$ and $t$ are two additional vertices.
The edges of $H$ are between every two vertices in $Y$, between $s$ and any vertex in $Y$, and between $t$ and every vertex in $Y$. In other words, $H$ is the complete split graph such that $Y$ induces a clique and $\{s,t\}$ is an independent set. An edge of $H$ between two vertices $u$ and $v$ in $Y$ is given a weight
$$
w_{u,v}=-\frac{1}{2}q_{u,v}=
\left\{
\begin{array}{rl}
0 & \mbox{ when } u=v\mbox{,}\\
1/2 & \mbox{ when }v\in{E_u}\mbox{, and}\\
1 & \mbox{ otherwise}\mbox{,}\\
\end{array}
\right.
$$
where $q_{u,v}$ is the non-positive coefficient of the monomial $x_ux_v$ in the expression of $q(x)$ from Proposition~\ref{CPV2.sec.1.prop.1}. We further denote
$$
\ell_u=2|C_u| - |D_u| + 2|F_u|\mbox{,}
$$
the coefficient of $x_u$ in the expression of $\ell (x)$ from Proposition~\ref{CPV2.sec.1.prop.1}, and
$$
r_u=\sum_{v\in Y} -w_{u,v}\mbox{.}
$$

For every edge of $H$ between $s$ and a vertex $v$ in $Y$, we define
$$
w_{s,v}= \max\{0, -r_v-\ell_v\}\mbox{,}
$$
and for every edge of $H$ between $t$ and a vertex $u$ in $Y$,
$$
w_{u,t}=\max\{0, r_u + \ell_u\}\mbox{.}
$$

Observe that all the edge weights thus defined are non-negative.

An $(s,t)$-cut of value $w$ in $H$ is a bipartition of the vertices of $H$ such that $s$ and $t$ are not in the same part and the weights of all the edges between the two parts sum to exactly $w$. Such a cut can be encoded by the point $x$ from $\{0,1\}^{|Y|}$ whose coordinate $x_u$ is set to $1$ when $u$ belongs to the part that contains $s$ and to $0$ when $u$ belongs to the part that contains $t$.

The following lemma states that minimizing the function $\ell(x)+q(x)=f(x)-c$ on $\{0,1\}^{|Y|}$ amounts to solving the minimum $(s,t)$-cut problem in $H$. Note that this indeed solves our problem, since the constant term $c$ in the expression of $f(x)$ can be safely ignored when it comes to minimizing.

\begin{lem}
The graph $H$ contains an $(s,t)$-cut of value $w$ if and only if there exists a point $x$ in $\{0,1\}^{|Y|}$ such that
$$
\sum_{v\in Y} w_{s,v} + \ell (x) + q(x)=w\mbox{.}
$$
\end{lem}
\begin{proof}
Consider an $(s,t)$-cut of value $w$ in $H$ and the point $x$ from $\{0,1\}^{|Y|}$ that encodes it. The value $w$ of the cut can be rewritten as
$$
w=\sum_{v\in Y} w_{s,v} (1-x_v) + \sum_{u\in Y}\sum_{v\in Y} w_{u,v} x_u(1-x_v)+\sum_{u\in Y} w_{u,t}x_u\mbox{.}
$$

By the above expression of $w_{s,v}$,
$$
\sum_{v\in Y} w_{s,v} (1-x_v)=\sum_{v\in Y}w_{s,v}-\sum_{v\in Y}\max \{0, -r_v-\ell_v\} x_v
$$
and by the expression of $w_{u,t}$,
$$
\sum_{u\in Y} w_{u,t}x_u=\sum_{u\in Y}\max\{0, r_u + \ell_u\}x_u\mbox{.}
$$

Since $\max\{0, r_u + \ell_u\}-\max \{0, -r_v-\ell_v\}=r_v+\ell_v$,
\begin{equation}\label{CPV2.sec.3.eq.1}
w=\sum_{v\in Y}w_{s,v}+\sum_{u\in{Y}}(r_u+\ell_u) x_u+\sum_{u\in Y}\sum_{v\in Y} w_{u,v} x_u(1-x_v)\mbox{.}
\end{equation}

Now observe that 
$$
\begin{array}{l}
\displaystyle\sum_{u\in Y}\sum_{v\in Y} w_{u,v} x_u(1-x_v)=\sum_{u\in Y}\left(\sum_{v\in Y}w_{u,v}\right)\!x_u-\sum_{u\in Y}\sum_{v\in Y} w_{u,v} x_ux_v\\[1.5\bigskipamount]
\hfill\displaystyle=-\sum_{u\in Y}r_ux_u+\frac{1}{2}\sum_{u\in Y}\sum_{v\in Y} q_{u,v} x_ux_v\mbox{.}\\
\end{array}
$$

As a consequence, (\ref{CPV2.sec.3.eq.1}) can be rewritten into
$$
w=\sum_{v\in Y}w_{s,v}+\ell(x)+q(x)
$$
and the desired result follows.
\end{proof}

Using state-of-the-art algorithms for dense instances of maximum flow~\cite{vandenBrandLeeLiuSaranurakSidfordSongWang2021}, we obtain that the problem can be solved in time $O(n^{2+o(1)})$. In fact, since the number of variables only depends on the number $q$ of vertices in $Q$, we can further refine the result as follows

\begin{thm}
Let $G$ be a complete split graph with $p+q$ vertices, $p$ of whose induce a clique in $G$ and the other $q$ form an independent set of $G$.
The rotation distance between two brooms on $G$ can be computed in time $O(p+q^{2+o(1)})$.
\end{thm}
  
\noindent{\bf Acknowledgment.}
Most of this paper was written while the first author was a visiting professor at the computer science department (LIPN) of the Université Paris 13 in the Fall of 2021 and 2022. The authors wish to thank the referees for their careful reading and useful comments.

\bibliography{BroomDistance}
\bibliographystyle{ijmart}

\end{document}